\documentclass{amsart}
\usepackage[T1]{fontenc}

\usepackage{amsmath,amssymb,epsfig,mathtools}
\usepackage{graphicx}
\usepackage[font=small,skip=0pt]{caption}
\usepackage{subcaption}
\usepackage{hyperref}
\usepackage{dsfont}
\usepackage{color}
\usepackage{enumerate}
\usepackage{url}
\usepackage{algorithm}
\usepackage{algpseudocode}
\usepackage[table]{xcolor}

\usepackage[export]{adjustbox}

\usepackage{graphicx}
\newtheorem{theorem}{Theorem}[section]
\newtheorem{lemma}[theorem]{Lemma}
\newtheorem{proposition}[theorem]{Proposition}
\newtheorem{corollary}[theorem]{Corollary}

\theoremstyle{definition}

\theoremstyle{definition}

\theoremstyle{remark}

\newcommand{\eqdef}{\ensuremath{\stackrel{\mbox{\upshape\tiny def.}}{=}}}

\newcommand{\mres}{\mathbin{\vrule height 1.6ex depth 0pt width
		0.13ex\vrule height 0.13ex depth 0pt width 1.3ex}}
\newcommand{\norm}[1]{\left\lVert#1\right\rVert}
\newcommand{\inner}[1]{\left\langle#1\right\rangle}

\newcommand{\prox}[1]{\mathop{\text{prox}_{#1}}}
\newcommand{\TV}{\mathop{\text{\rm TV}}}
\newcommand{\BV}{\mathop{\text{\rm BV}}}
\newcommand{\Per}{\mathop{\text{\rm Per}}}
\def\dd{{\rm d}}

\def\argmin{\mathop{\rm argmin}}

\def\ddiv{\mathop{\rm div}}

\def\inf{\mathop{\rm inf}}

\def\min{\mathop{\rm min}}

\begin{document}
%
\title[The TV-Wasserstein problem]{The Total Variation-Wasserstein problem:\\ a new derivation of the Euler-Lagrange equations}
%
%

%
%
%

\author{Antonin Chambolle}
\email{chambolle@ceremade.dauphine.fr}
\author{Vincent Duval}
\email{vincent.duval@inria.fr}
\author{João-Miguel Machado}
\email{joao-miguel.machado@ceremade.dauphine.fr}
\address{CEREMADE, CNRS, Universit\'e Paris-Dauphine, Universit\'e PSL, 75016 PARIS, FRANCE}
\address{Inria}
\thanks{The authors are members of the Inria MOKAPLAN team. V.D. gratefully acknowledges support from the Agence Nationale de la Recherche (CIPRESSI, ANR-19-CE48-0017-01). }

\maketitle              
\begin{abstract}
    In this work we analyze the Total Variation-Wasserstein minimization problem. We propose an alternative form of deriving optimality conditions from the approach of~\cite{carlier2019total}, and as result obtain further regularity for the quantities involved. In the sequel we propose an algorithm to solve this problem alongside two numerical experiments. 
\end{abstract}
%
%
%
\section{Introduction}
The Wasserstein gradient flow of the total variation functional has been studied in a series of recent papers~\cite{burger2012regularization,benning2013primal,carlier2019total}, for applications in image processing. In the present paper, we revisit the work of Carlier \& Poon~\cite{carlier2019total} and derive Euler-Lagrange equations for the problem: given $\Omega \subset \mathbb{R}^d$ open, bounded and convex, $\tau > 0$ and an absolutely continuous probability measure $\rho_0 \in \mathcal{P}(\Omega)$
\begin{equation}\label{TV-W}
	\tag{TV-W}
	\inf_{\rho \in \mathcal{P}(\Omega)}
	\TV(\rho) 
    +
    \frac{1}{2\tau}W^2_2(\rho_0, \rho), 
\end{equation}
where $\tau$ is interpreted as a time discretization parameter for an implicit Euler scheme, as we shall see below. 

The {\em total variation functional} of a Radon measure $\rho \in \mathcal{M}(\Omega)$ is defined as
\begin{equation}\label{TV}
	\tag{TV}
	\TV(\rho) = \sup 
	\left\{
	\int_{\Omega} \ddiv z \dd \rho : z \in C^1_{c}\left(\Omega; \mathbb{R}^N\right), \norm{z}_{\infty} \le 1 \right\}, 
\end{equation}
which is not to be mistaken in this paper with the {\em total variation measure} $|\mu|$ of a Radon measure $\mu$ or its {\em total variation norm} $|\mu|(\Omega)$. We call $\text{BV}(\Omega)$ the subspace of functions $u \in L^1(\Omega)$ whose weak derivative $Du$ is a {\em finite Radon measure}. It can also be characterized as the $L^1$ functions such that $\TV(u)<\infty$, where $\TV(u)$ should be understood as in~\eqref{TV} with the measure $u\mathcal{L}^d \mres \Omega$, and it holds that $\TV(u) = |Du|(\Omega)$. As $\text{BV}(\mathbb{R}^d) \xhookrightarrow{} L^{\frac{d}{d-1}}(\mathbb{R}^d)$, solutions to~\eqref{TV-W} are also absolutely continuous w.r.t.~the Lebesgue measure. Therefore, w.l.o.g.~we can minimize on $L^{\frac{d}{d-1}}(\Omega)$, which is a reflexive Banach space. In addition, a function $\rho$ will have finite energy only if $\rho \in \mathcal{P}(\Omega)$.   

The data term is given by the Wasserstein distance, defined through the value of the optimal transportation problem (see~\cite{santambrogio2015optimal})
\begin{equation}\label{wasserstein_distance}
    W_2^2(\mu, \nu) 
    \eqdef 
    \min_{\gamma \in \Pi(\mu, \nu)} 
    \int_{\Omega \times \Omega} |x - y|^2\dd \gamma 
    =
    \sup_{\substack{ 
        \varphi, \psi \in C_b(\Omega)\\ 
        \varphi \oplus \psi \le |x - y|^2
        }
    }
    \int_{\Omega}\varphi \dd \mu + \int_{\Omega} \psi \dd \nu,
\end{equation}
where the minimum is taken over all the probability measures on $\Omega\times \Omega$ whose marginals are $\mu$ and $\nu$. An optimal pair $(\varphi, \psi)$ for the dual problem is referred to as Kantorovitch potentials.

Using total variation as regularization was suggested in~\cite{rudin1992nonlinear} with a $L^2$ data term for the~{\em Rudin-Osher-Fatemi problem}
\begin{equation}\label{ROF}
	\tag{ROF}
	\inf_{u \in L^2(\Omega)}  
    \TV(u) 
    +
    \frac{1}{2\lambda}\norm{u - g}^2_{L^2(\Omega)}, 
\end{equation}
see~\cite{chambolle2010introduction} for an overview. Other data terms were considered to better model the oscillatory behavior of the noise~\cite{meyer2001oscillating,lieu2008image}. More recently Wasserstein energies have shown success in the imaging community~\cite{cuturi2018semidual}, the model~\eqref{TV-W} being used for image denoising in~\cite{benning2013primal,burger2012regularization}.

Existence and uniqueness of solutions for~\eqref{TV-W} follow from the direct method in the calculus of variations, and the strict convexity of $W_2^2(\rho_0,\cdot)$ whenever $\rho_0$ is absolutely continuous, see~\cite[Prop.~7.19]{santambrogio2015optimal}. However, it is not easy to compute the subdifferential of the sum, which makes the derivation of the Euler-Lagrange equations not trivial.

In~\cite{carlier2019total}, the authors studied the gradient flow scheme defined by the successive iterations of~\eqref{TV-W}, and following the seminal work~\cite{jordan1998variational} they showed that, in dimension 1 as the parameter $\tau \to 0$, the discrete scheme converges to the solution of a fourth order PDE.~They used an entropic regularization approach, followed by a $\Gamma$-convergence argument, to derive an Euler-Lagrange equation, which states that there exists a Kantorovitch potential $\psi_1$ coinciding with some $\ddiv z \in \partial \TV(\rho_1)$ in the set $\{\rho_1 > 0\}$. On $\{\rho_1 = 0\}$, these quantities are related through a bounded Lagrange multiplier $\beta$ associated with the nonnegativity constraint $\rho_1 \ge 0$. 

In this work we propose an alternative way to derive the Euler-Lagrange equations which relies on the well established properties of solutions of~\eqref{ROF} and shows further regularity of the quantities $\ddiv z, \beta$.
\begin{theorem}\label{theorem.Euler-Lagrange_TVJKO}
	For any $\rho_0 \in L^1(\Omega)\cap \mathcal{P}(\Omega)$, let $\rho_1$ be the unique minimizer of~\eqref{TV-W}. The following hold.
	\begin{enumerate}
		\item There is a vector field $z \in H^1_0(\ddiv; \Omega)\cap L^{\infty}(\Omega; \mathbb{R}^d)$ and a Lagrange multiplier $\beta \ge 0$ such that 
		\begin{equation}
			\label{Euler_Lagrange.TV-JKO}
			\tag{TVW-EL}
			\left\{
			\begin{array}{rl}
				\ddiv z + \frac{\psi_1}{\tau} = \beta,& \text{ a.e. in $\Omega$}\\
				z\cdot\nu = 0,& \text{ on $\partial\Omega$}\\
				\beta\rho_1 = 0, & \text{ a.e. in $\Omega$}\\ 
				z\cdot D\rho_1 = |D\rho_1|,& \norm{z}_\infty \le 1,
			\end{array}
			\right.
			\end{equation}
			where $\psi_1$ is a Kantorovitch potential associated with $\rho_1$. 
			\item The Lagrange multiplier $\beta$ is the unique solution to~\eqref{ROF} with $\lambda = 1$ and $g = \psi_1/\tau$.
		\item The functions $\ddiv z, \psi_1$ and $\beta$ are Lipschitz continuous. 
	\end{enumerate} 
\end{theorem}

\section{The Euler-Lagrange equation}
Let $X$ and $X^\star$ be duality-paired spaces and $f:X\to \mathbb{R}\cup \{\infty\}$ be a convex function, the subdifferential of $f$ on $X$ is given by
\begin{equation}
	\partial_X f(u)
	\eqdef 
	\left\{
		p \in X^\star : 
		f(v) \ge f(u) + \inner{p, v - u}, \text{ for all } v \in X
	\right\}.
\end{equation} 
In order to derive optimality conditions for~\eqref{TV-W} we will need some properties of the subdifferential of $\TV$ and of~\eqref{ROF}.

\begin{proposition}{\cite{bredies2016pointwise,chambolle2010introduction,mercier2018continuity}}\label{proposition.ROF_TV_properties}
	If $u \in \BV(\Omega)$, then the subdifferential of $\TV$ at $u$ assumes the form
	\[
		\partial_{L^2}\TV(u)
		= 	
		\left\{
			p \in L^2(\Omega): 
			\begin{array}{c}
				p = -\ddiv z, \ z \in H^1_0(\ddiv; \Omega),\\ 
				\norm{z}_\infty \le 1, \ |Du| = z\cdot Du
			\end{array}
		\right\}.
	\]

	If $p \in \partial_{L^2} \TV(u)$, then 
	\begin{equation*}
		p \in \partial_{L^2} \TV(u^+),\quad
		-p \in \partial_{L^2} \TV(u^-).
	\end{equation*}

	If in addition $u$ solves~\eqref{ROF}, then 
	\begin{enumerate}
		\item $u^+$ solves~\eqref{ROF} with the constraint $u \ge 0$;
		\item it holds that
		\begin{equation}\label{EL_ROF}
			0 \in \frac{u-g}{\lambda}+ \partial_{L^2}\TV(u),
		\end{equation}
		and conversely, if $u$ satisfies~\eqref{EL_ROF}, $u$ minimizes~\eqref{ROF};
		\item  for $\Omega$ convex, if $g$ is uniformly continuous with modulus of continuity $\omega$, then $u$ has the same modulus of continuity.
	\end{enumerate}
\end{proposition}
In the previous proposition, we recall that $H^1_0(\ddiv; \Omega)$ denotes the closure of $C^\infty_c(\Omega; \mathbb{R}^d)$ with respect to the norm $\norm{z}_{H^1(\ddiv)}^2 = \norm{z}_{L^2(\Omega)}^2 + \norm{\ddiv z}_{L^2(\Omega)}^2$. These properties are known in the literature of~\eqref{ROF}, but for the reader's convenience they are proven in the appendix~\ref{appendix.ROF}.

Unless otherwise stated, we consider in the sequel $X=L^{\frac{d}{d-1}(\Omega)}$, $X^\star = L^d(\Omega)$ and we drop the index $X$ in the notation $\partial_X$. 
Under certain regularity conditions, one can see the Kantorovitch potentials as the first variation of the Wasserstein distance,~\cite{santambrogio2015optimal}. As a consequence, Fermat's rule $0 \in \partial\left(W_2^2(\rho_0, \cdot) + \TV(\cdot)\right)(\rho_1)$ assumes the following form. 
\begin{lemma}\label{lemma.almost_euler_lagrange}
	Let $\rho_1$ be the unique minimizer of~\eqref{TV-W}, then there exists a Kantorovitch potential $\psi_1$ associated to $\rho_1$ such that 
	\begin{equation}
	\label{subdifferential_almost_inclusion}
		-\frac{\psi_1}{\tau} \in \partial \left(\TV + \chi_{\mathcal{P}\left(\Omega\right)} \right)(\rho_1). 
	\end{equation}
\end{lemma}
\begin{proof}
	For simplicity, we assume $\tau = 1$. Take $\rho \in \BV(\Omega)\cap \mathcal{P}(\Omega)$ and define $\rho_t \eqdef \rho + t(\rho_1 - \rho)$. Since $\overline{\Omega}$ is compact, the $\sup$ in~\eqref{wasserstein_distance} admits a maximizer~\cite[Prop.~1.11]{santambrogio2015optimal}. Let $\varphi_t, \psi_t$ denote a pair of Kantorovitch potentials between $\rho_0$ and $\rho_t$. From the optimality of $\rho_1$ it follows
	\begin{align*} 
		&\frac{1}{2}W_2^2(\rho_0, \rho_1) 
		+ 
		\TV(\rho_1)
		\le 
		\int_\Omega \varphi_t \dd \rho_0 + \int_\Omega \psi_t \dd \rho_t + \TV(\rho_t)\\
		&\le
		\int_\Omega \varphi_t \dd \rho_0 + \int_\Omega \psi_t \dd \rho_1 + \TV(\rho_1)
		+ (1-t)\left(\int_\Omega \psi_t \dd(\rho - \rho_1)  + \TV(\rho) - \TV(\rho_1)\right)\\
		&\le 
		\frac{1}{2}W_2^2(\rho_0, \rho_1) 
		+ 
		\TV(\rho_1)
		+
		(1-t)\left(\int_\Omega \psi_t \dd(\rho - \rho_1)  + \TV(\rho) - \TV(\rho_1)\right).
	\end{align*}
	 Hence, $-\psi_t \in \partial \left(\TV + \chi_{\mathcal{P}\left(\Omega\right)} \right)(\rho_1)$ for all $t \in (0,1)$. Notice that as the optimal transport map from $\rho_0$ to $\rho_t$ is given by $T_t = \text{id} - \nabla \psi_t$ and assumes values in the bounded set $\Omega$, the family $\left(\psi_t\right)_{t \in [0,1]}$ is uniformly Lipschitz so that by Arzelà-Ascoli's Theorem $\psi_t$ converges uniformly to $\psi_1$ as $t$ goes to $1$ (see also~\cite[Thm.~1.52]{santambrogio2015optimal}). Therefore, $-\psi_1 \in \partial \left(\TV + \chi_{\mathcal{P}\left(\Omega\right)} \right)(\rho_1)$.
\end{proof}
With these results we can prove Theorem~\ref{theorem.Euler-Lagrange_TVJKO}. 

\begin{proof}[Proof of Theorem~\ref{theorem.Euler-Lagrange_TVJKO}]
	Here, to simplify, we still assume $\tau = 1$. The subdifferential inclusion~\eqref{subdifferential_almost_inclusion} is conceptually the Euler-Lagrange equation for~\eqref{TV-W}, however it can be difficult to verify the conditions for direct sum between subdifferentials and give a full characterization. Therefore, for some arbitrary $\rho \in \mathcal{M}_+(\Omega)$ and $t>0$, set
\begin{equation*}
	\rho_t = \frac{\rho_1 + t(\rho - \rho_1)}{1 + t\alpha }, \text{ where } \alpha = \int_{\Omega}\dd(\rho - \rho_1).
\end{equation*}
Now $\rho_t$ is admissible for the subdifferential inequality and using the positive homogeneity of $\TV$ we can write 
\begin{equation*}
	\TV(\rho_1) - \int_{\Omega}\psi_1\dd \left(\rho_t - \rho_1 \right) \le \frac{\TV(\rho_1) + t\left(\TV(\rho) - \TV(\rho_1)\right)}{1 + t\alpha}.
\end{equation*}
After a few computations we arrive at 
$
	\TV(\rho) \ge \TV(\rho_1) + \int_{\Omega}(C - \psi_1)\dd(\rho - \rho_1),
$
where $C =  \TV(\rho_1) + \int_{\Omega}\psi_1 \dd \rho_1$. Notice that $(\phi + C, \psi - C)$ remains an optimal potential. So we can replace $\psi_1$ by $\psi_1 - C$, and obtain that for all $\rho \ge 0$ the following holds 
\begin{equation}
	\label{psi_1_ineq_positivemeasures}
	\TV(\rho) \ge \TV(\rho_1) + \int_{\Omega} - \psi_1\dd(\rho - \rho_1), \text{ with } \TV(\rho_1) = \int_{\Omega} - \psi_1\dd\rho_1.
\end{equation}
In particular, this means $-\psi_1 \in \partial \left(\TV + \chi_{\mathcal{M}_+\left(\Omega\right)}\right)(\rho_1)$ and $\rho_1$ is optimal for
\begin{equation}
	\label{psi_problem}
	\inf_{\rho \ge 0} \mathcal{E}(\rho) := \TV(\rho) + \int_{\Omega}\psi_1(x)\rho(x)\dd x.
\end{equation}

This suggests a penalization with an $L^2$ term {\em e.g.}
\begin{equation}
	\label{psi_ROF_problem}
	\inf_{u \in L^2(\Omega)} \mathcal{E}_t(u) := \TV(u) + \int_{\Omega}\psi_1(x)u(x)\dd x + \frac{1}{2t}\int_{\Omega}|u -\rho_1|^2\dd x
\end{equation}
which is a variation of~\eqref{ROF} with $g = \rho_1 - t\psi_1$. In order for~\eqref{psi_ROF_problem} to make sense, we need $\rho_1\in L^2(\Omega)$, which is true if $\rho_0$ is $L^\infty$ since then~\cite[Thm. 4.2]{carlier2019total} implies $\rho_1\in L^\infty$. Suppose for now that $\rho_0$ is a bounded function.

Let $u_t$ denote the solution of~\eqref{psi_ROF_problem}, from Prop.~\ref*{proposition.ROF_TV_properties} if $u_t$ solves~\eqref{psi_ROF_problem}, then $u_t^+$ solves the same problem with the additional constraint that $u \ge 0$, see~\cite[Lemma~A.1]{chambolle2004algorithm}. As $\rho_1 \ge 0$ we can compare the energies of $u^+_t$ and $\rho_1$ and obtain the following inequalities
\begin{align*}
	\mathcal{E}(\rho_1) \le \mathcal{E}(u_t^+) \text{ and } 
	\mathcal{E}_t(u_t^+) \le \mathcal{E}_t(\rho_1).
\end{align*} 
Summing both inequalities yields
\begin{equation}
\int_{\Omega} |u_t^+ - \rho_1|^2\dd x \le 0, \text{ therefore } u_t^+ = \rho_1 \text{ a.e. on $\Omega$.}
\end{equation}
In particular, we also have that $u_t \le \rho_1$. But as $u_t$ solves a~\eqref{ROF} problem, the optimality conditions from Prop.~\ref{proposition.ROF_TV_properties} give
\begin{equation}\label{el_almost}
	\beta_t -\psi_1 \in \partial_{L^2} \TV(u_t),
	\text{ where } 
	\beta_t
	\eqdef 
	\frac{\rho_1 - u_t}{t} \ge 0.
\end{equation}
Notice from the characterization of $\partial_{L^2}\TV(\cdot)$ that $\partial_{L^2}\TV(u) \subset \partial_{L^2}\TV(u^+)$. Since $u_t^+ = \rho_1$, we have that 
\begin{equation}\label{el_full}
	\beta_t -\psi_1 \in \partial_{L^2} \TV(\rho_1),
\end{equation}
which proves~\eqref{Euler_Lagrange.TV-JKO}.

Now we move on to study the family $\left(\beta_t\right)_{t > 0}$. Since $\rho_1 = u_t^+$, by definition $\beta_t = u_t^-/t$ and using the fact that $\partial_{L^2}\TV(u) \subset \partial_{L^2}\TV(u^-)$ in conjunction with equation~\eqref{el_almost}, it holds that
\begin{equation}\label{eq:subdiff_beta}
	\psi_1 - \beta_t \in \partial_{L^2}\TV(\beta_t).
\end{equation}
But then, from Prop.~\ref{proposition.ROF_TV_properties}, $\beta_t$ solves~\eqref{ROF} with $g = \psi_1$ and $\lambda = 1$. As this problem has a unique solution, the family $\left\{\beta_t\right\}_{t>0} = \{\beta\}$ is a singleton. 

Since $\Omega$ is convex, and we know that the Kantorovitch potentials are Lipschitz continuous, cf.~\cite{santambrogio2015optimal}, so $\beta$, as a solution of~\eqref{ROF} with Lipschitz data $g = \psi_1$, is also Lipschitz continuous with the same constant, following~\cite[Theo.~3.1]{mercier2018continuity}. 

But from~\eqref{el_full} and the characterization of the subdifferential of $\TV$, there is a vector field $z$ such that $z\cdot D\rho_1 = |D\rho_1|$ such that
\[
	\beta - \psi_1 = \ddiv z,
\]
and as a consequence $\ddiv z$ is also Lipschitz continuous, with constant at most twice the constant of $\psi_1$.

In the general case of $\rho_0 \in L^1(\Omega)$, define $\rho_{0,N} \eqdef c_N(\rho_0 \wedge N)$ for $N \in \mathbb{N}$, where $c_N$ is a renormalizing constant. Then $\rho_{0,N} \in L^\infty(\Omega)$ and $\rho_{0,N} \xrightarrow[N \to \infty]{L^1} \rho_0$. Let $\rho_{1,N}$ denote the unique minimizer of~\eqref{TV-W} with data term $\rho_{0,N}$, we can assume that $\rho_{1,N}$ w-$\star$ converges to some $\tilde \rho$. Then for any $\rho \in \mathcal{P}(\Omega)$ we have 
	\[
		\TV(\rho_{1,N}) + \frac{1}{2\tau} W_2^2(\rho_{0,N}, \rho_{1,N}) \le
		\TV(\rho) +       \frac{1}{2\tau} W_2^2(\rho_{0,N}, \rho). 
	\]
	Passing to the limit on $N\to \infty$ we have that $\tilde \rho$ is a minimizer and from uniqueness it must hold that $\tilde \rho = \rho_1$.

	Hence, consider the functions $z_N, \psi_{1,N}, \beta_N$ that satisfy~\eqref{Euler_Lagrange.TV-JKO} for $\rho_{1,N}$. Up to a subsequence, we may assume that $z_N$ converges weakly-$\star$ to some $z \in L^\infty(\Omega;\mathbb{R}^d)$. Since $\psi_{1,N}$, $\beta_N$ and $\ddiv z_N$ are Lipschitz continuous with the same Lipschitz constant for all $N$, by Arzelà-Ascoli, we can assume that $\psi_{1,N}, \beta_N$ and $\ddiv z_N$ converge uniformly to Lipschitz functions $\psi_1, \beta, \ddiv z = \beta - \psi_1$. In addition, passing to the limit in~\eqref{eq:subdiff_beta}, we find that $\beta$ solves~\eqref{ROF} for $\lambda = 1$ and $g = \psi_1$.

	Since $\beta_N$ converges uniformly and $\rho_{1,N}$ converges w-$\star$ we have
	\[
		0 = \lim_{N \to \infty} \int_\Omega \beta_N \rho_{1,N} \dd x
		= \int_\Omega \beta \rho_{1} \dd x,
	\]
	and hence $\beta \rho_1 = 0$ a.e.~in $\Omega$ since both are nonnegative. In addition, $\psi_1$ is a Kantorovitch potential associated to $\rho_1$ from the stability of optimal transport (see~\cite[Thm. 1.52]{santambrogio2015optimal}).
	From the optimality of $\rho_{1,N}$ it holds that
	\[
		\TV(\rho_{1,N}) + \frac{1}{2\tau} W_2^2(\rho_{0,N}, \rho_{1,N}) \le
		\TV(\rho_1)     + \frac{1}{2\tau} W_2^2(\rho_{0,N}, \rho),
	\] 
	so that $\lim \TV(\rho_{1,N}) \le \TV(\rho_1)$. Changing the roles of $\rho_1$ and $\rho_{1,N}$ we get an equality. So it follows that
	\[
		\int_\Omega (\beta - \psi_1)\rho_1\dd x 
		= 
		\lim_{N \to \infty}
		\int_\Omega (\beta_N - \psi_{1,N})\rho_{1,N}\dd x
		=  
		\lim_{N \to \infty}
		\TV(\rho_{1,N})
		= 
		\TV(\rho_1),
	\]
	Since $\TV$ is 1-homogeneous we conclude that $\beta - \psi_1 \in \partial \TV(\rho_1)$.
\end{proof}

We say $E$ is a set of finite perimeter if the indicator function $\mathds{1}_E$ is a BV function, and we set $\Per(E) = \TV(\mathds{1}_E)$. As a byproduct of the previous proof we conclude that the level sets $\{\rho_1 > s\}$ are all solutions to the same prescribed curvature problem.
\begin{corollary}\label{corollary.rho_1_levelset_props}
	The following properties of the level sets of $\rho_1$ hold.
	\begin{enumerate}
		\item For $s > 0$ and $\psi_1$ in~\eqref{Euler_Lagrange.TV-JKO}
		\[
		\{\rho_1 > s\} \in \argmin_{E \subset \Omega} \Per(E; \Omega) + \frac{1}{\tau}\int_{E}\psi_1\dd x
		\]
		\item $ \partial \{\rho_1 > s\} \setminus \partial^* \{\rho_1 > s\}$ is a closed set of Hausdorff dimension at most $d - 8$, where $\partial^*$ denotes the reduced boundary of a set, see~\cite{ambrosio2000functions}. In addition, $\partial^* \{\rho_1 > s\}$ is locally the graph of a function of class $W^{2,q}$ for all $q < +\infty$.
	\end{enumerate}
\end{corollary}
\begin{proof}
	For simplicity take $\tau = 1$. Inside the set $\{\rho_1 > s\}$, for $s > 0$, we have $-\psi_1 = \ddiv z$, so from the definition of the perimeter we have
	\[
		\int_{\{\rho_1 > s\}}-\psi_1\dd x = \int_{\{\rho_1 > s\}} \ddiv z \dd x \le \Per\left(\{\rho_1 > s\}\right).
	\]
	So using the fact that $\TV(\rho_1) =\int_\Omega-\psi_1\dd x$, the coarea formula and Fubini's Theorem give
	\begin{align*}
		\int_{0}^{+\infty}\text{Per}(\mathds{1}_{\{\rho_1 > s\}})\dd s = \int_{\Omega}-\psi_1\int_0^{\rho_1(x)}\dd s\dd x = \int_0^{+\infty}\int_{\{\rho_1 > s\}}-\psi_1 \dd x\dd s.
	\end{align*}
	Hence, $\text{Per}(\{\rho_1 > s\}) = \int_{\{\rho_1 > s\}}-\psi_1 \dd x$ for {\em a.e.} $s>0$. But as $\beta \psi_1 = 0$ a.e., we have $-\psi_1 = \ddiv z$ in $\{\rho_1>s\}$, so that $-\psi_1 \in \partial \TV(\mathds{1}_{\{\rho_1 > s\}})$ for {\em a.e.} $s>0$; and by a continuity argument, for all $s>0$. The subdifferential inequality with $\mathds{1}_E$ gives
	\begin{equation}\label{pb_sublevelsets_rho1}
		\{\rho_1 > s\} \in \argmin_{E \subset \Omega} \text{Per}(E) + \int_E \psi_1(x) \dd x.
	\end{equation}

	Item (2) follows directly from the properties of \eqref{ROF}, see~\cite{chambolle2010introduction}, since $\rho_1 = u^+$, where $u$ solves a problem~\eqref{ROF}. 
\end{proof}

\section{Numerical Experiments}\label{section.experiments}
We solve~\eqref{TV-W} for an image denoising application using a Douglas-Rachford algorithm~\cite{combettes2011proximal} with Halpern acceleration~\cite{contreras2022optimal}, see table~\ref{algorithm.douglas_rachford}. For this we need subroutines to compute the prox operators defined, for a given $\lambda > 0$, as
\begin{align}
	\prox{\lambda\TV}(\bar \rho)
	&\eqdef 
	\argmin_{\rho \in L^2(\Omega)}
	\TV(\rho) + \frac{1}{2\lambda} \norm{\rho - \bar \rho}_{L^2(\Omega)}^2,\\
	\prox{\lambda W_2^2}(\bar \rho)
	&\eqdef 
	\argmin_{\rho \in L^2(\Omega)}
	\frac{1}{2\tau}W_2^2(\rho_0, \rho) + \frac{1}{2\lambda} \norm{\rho - \bar \rho}_{L^2(\Omega)}^2.
\end{align}

We implemented the $\prox{}$ of $\TV$ with the algorithm from~\cite{Condat-dtv}, modified to account for Dirichlet boundary conditions. From~\cite[Theo.~2.4]{chambolle2021learn} it is consistent with the continuous total variation. The $\prox{}$ of $W_2^2$ is computed by expanding the $L^2$ data term as
\begin{align*}
	\prox{\lambda W_2^2}(\bar \rho)
	&=
	\argmin_{\rho \in L^2(\Omega)}
	\frac{1}{2\tau}W_2^2(\rho_0, \rho) + 
	\frac{1}{2\lambda}\int_\Omega \rho^2\dd x
	+
	\int_\Omega \rho\underbrace{\left(-\frac{\bar \rho}{\lambda}\right)}_{= V}\dd x
	+ 
	\underbrace{\frac{1}{2\lambda}\bar\rho^2\dd x}_{cst}\\
	&=
	\argmin_{\rho \in L^2(\Omega)}
	\frac{1}{2\tau}W_2^2(\rho_0, \rho) + 
	\frac{1}{2\lambda}\int_\Omega \rho^2\dd x
	+
	\int_\Omega \rho V\dd x,
\end{align*}
which is one step of the Wasserstein gradient flow of the porous medium equation $\partial_t \rho_t = \lambda^{-1}\Delta(\rho_t^2) +\ddiv \left(\rho_t \nabla V\right)$, where the potential is $V = -\bar \rho/\lambda$, see~\cite{santambrogio2015optimal,jacobs2021back}. To compute it we have used the back-n-forth algorithm from~\cite{jacobs2021back}.
\begin{algorithm}
	\caption{Halpern accelerated Douglas-Rachford algorithm}\label{algorithm.douglas_rachford}
	\begin{algorithmic}
	\State $\beta_0 \gets 0$
	\State $x_0 \gets $ Initial Image
	\While{$n \ge 0$}
		\State  $y_n \gets \prox{\lambda\TV}(x_n)$
		\State $\lambda_n \in [\varepsilon, 2 -\varepsilon]$
		\State  $z_n \gets x_n + \lambda_n \left(\prox{\lambda W_2^2}(2y_n - x_n) - y_n\right)$ 
		\State  $\beta_n \gets \frac{1}{2}\left(1 + \beta_{n-1}^2\right)$ \Comment{Optimal constants for Halpern acceleration from~\cite{contreras2022optimal}}
		\State $x_{n+1} \gets (1-\beta_n) x_0 + \beta_n z_n$ 
	\EndWhile
	\end{algorithmic}
\end{algorithm} 

\subsection{Evolution of balls}
Following~\cite{carlier2019total}, in dimension $1$, whenever the initial measure is uniformly distributed over a ball, the solutions remain balls. In $\mathbb{R}^d$, one can prove this remains true. If $\rho_0$ is uniformly distributed over a ball of radius $r_0$, then the solution to~\eqref{TV-W} is uniformly distributed in a ball of radius $r_1$ solving 
\[
	r_1^2(r_1 - r_0) = r_0^2(d+2)\tau,
\]
see appendix~\ref{appendix.Ndim} for a proof of this claim.
\begin{figure}
	\centering
	\includegraphics[scale=.15, trim = 0 0 0 0cm]{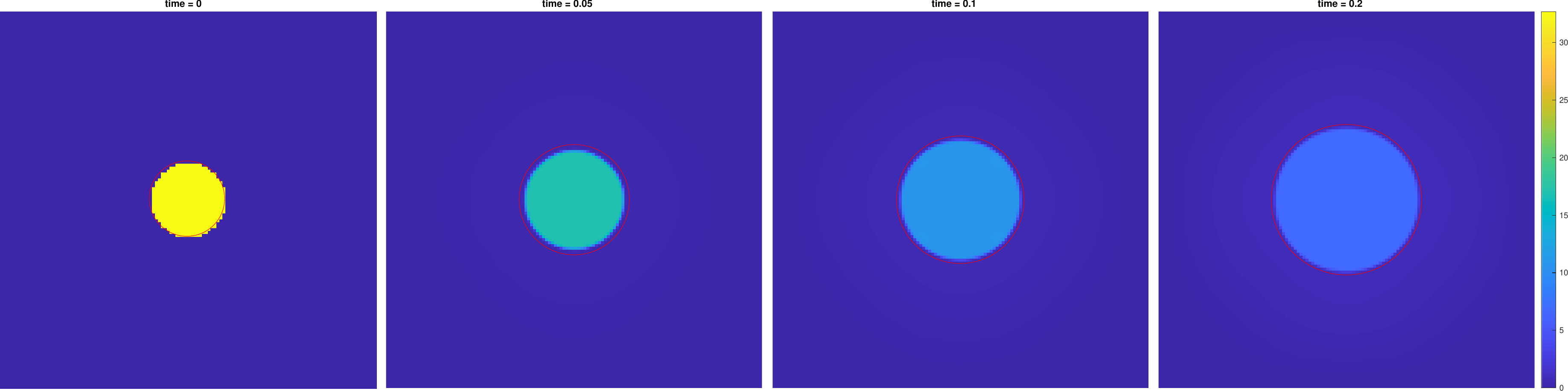}
	\caption{Evolution of circles: from left to right initial condition and solutions for $\tau = 0.05, 0.1,0.2$. The red circles correspond to the theoretical radius.}
	\label{fig:evolution_balls}
\end{figure}

\subsection{Reconstruction of dithered images}
In this experiment we use model~\eqref{TV-W} to reconstruct dithered images. In $\mathcal{P}(\mathbb{R}^2)$ the dithered image is a sum of Dirac masses, so the model~\eqref{TV-W} outputs a new image which is close in the Wasserstein topology, but with small total variation. In Figure~\ref{fig:dithering} below, we compared the result with the reconstruction given by~\eqref{ROF}, both with a parameter $\tau = 0.2$. Although the classical~\eqref{ROF} model was able to create complex textures, these remain granulated, whereas the~\eqref{TV-W} model is able to generate both smooth and complex textures. 
\begin{figure}
	\centering
\includegraphics[width=\linewidth,clip,trim= 0 0 0 0cm ]{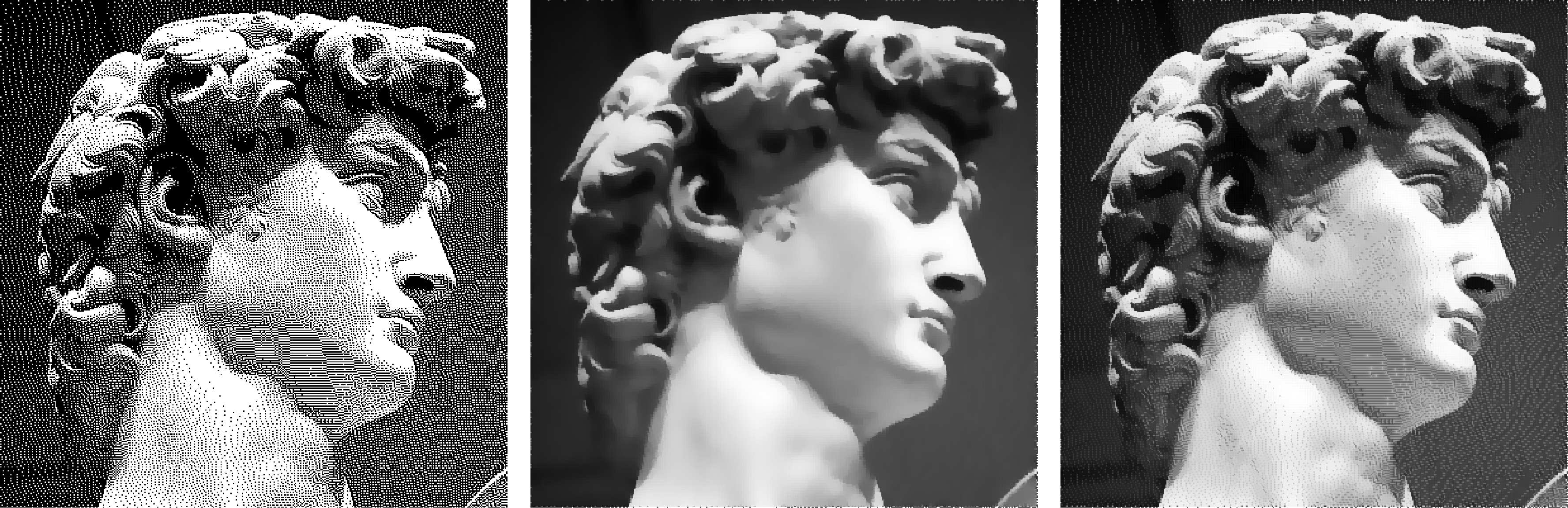}
	\caption{Dithering reconstruction problem. From left to right: Dithered image, TV-Wasserstein and ROF results.}
	\label{fig:dithering}
\end{figure}

\section{Conclusion}
In this work we revisited the TV-Wasserstein problem. We showed how it can be related to the classical~\eqref{ROF} problem and how to exploit this to derive the Euler-Lagrange equations, obtaining further regularity. We proposed a Douglas-Rachford algorithm to solve it and presented two numerical experiments: the first one being coherent with theoretical predictions and the second being an application to the reconstruction of dithered images. 

\subsubsection*{Acknowledgements} The second author acknowledges support from the ANR CIPRESSI project, grant 19-CE48-0017-01 of the French Agence Nationale de la Recherche. The third author acknowledges the financial support from FMJH concerning his master thesis, when this work took place.

\appendix
\section{Properties of (ROF)}\label{appendix.ROF}
\subsection{The Rudin-Osher-Fatemi (ROF) problem}

In this appendix we prove some properties about problem~\eqref{ROF} and the total variation functional that are used throughout the text. \if{
It is easy to see that \eqref{ROF} has a unique solution since the functional $u \mapsto \TV(u) + \norm{u - g}^2_{L^2(\Omega)}$ is strongly convex and l.s.c, a standard application of the direct method of the calculus of variations gives the result. In particular, this means that if we find the Euler-Lagrange equations for this problem, these will be necessary and sufficient for optimality. 

Therefore, let $u$ be such minimizer and $v \in L^2(\Omega)$ arbitrary. Then comparing the energies of $u$ and $v$ we have 
\begin{align*}
	\lambda\left(\TV(v) - \TV(u)\right) 
	&\ge \frac{1}{2}\int_{\Omega} \left((u - g)^2 - (v - g)^2\right)\dd x \\
	&= \int_{\Omega}(v - u)(g - u)\dd x - \frac{1}{2}\int_{\Omega}(u - v)^2\dd x.
\end{align*}
Now taking $u + t(v - u)$, for $t \in [0,1]$, as test function in the previous inequality we obtain that
\begin{equation}
	\lambda\left(\TV(u + t(v - u)) - \TV(u)\right) - t\int_{\Omega}(v - u)(g - u)\dd x \ge - \frac{t^2}{2}\int_{\Omega}(u - v)^2\dd x,
\end{equation}
which implies that 
\begin{equation}
	\lambda\left(\TV(u + t(v - u)) - \TV(u)\right) - t\int_{\Omega}(v - u)(g - u)\dd x \ge 0.
\end{equation}
Taking $t = 1$, we actually obtain the following Euler-Lagrange equation 
\begin{equation}
	\frac{g - u}{\lambda} \in \partial \TV(u).
\end{equation}
To extract further information, we need to characterize the subdifferential of $\TV$, which we will do in the next section in the more general case of one-homogeneous functionals. 
}\fi We shall start with general properties of 1-homogeneous functionals and use it to derive the characterization of the subdifferential of $\TV$ from Proposition~\ref{proposition.ROF_TV_properties}. 

\begin{lemma}{\cite[Thm.~7.57]{guide2006infinite}}\label{lemma.subdifferential_one_homo_functional}
	Let $X$ be a reflexive Banach space and $J:X \to \mathbb{R}\cup \{+\infty\}$ be a convex, positively one-homogeneous functional, {\em i.e.}~$J(\lambda u) = |\lambda|J(u)$ for all $\lambda \in \mathbb{R}$ and $u \in X$. Then 
	\begin{equation}
	\partial J(u) = \left\{
	p \in \partial J(0): \inner{p, u} = J(u)
	\right\}.
	\end{equation}
	In particular, if  $J$ is the support function: 
	\begin{equation}
		J(u) = \sup_{p \in C} \inner{p, u}, 
	\end{equation}
	of a set $C$ in the dual space $X^\star$,
	then the subdifferential $\partial J(0) = \overline{\textup{conv}C}$.
\end{lemma}
\if{
\begin{proof}
	Setting $A:= \left\{
	p \in \partial J(0): \inner{p, u} = J(u)
	\right\}$, this set is contained in $\partial J(u)$ since $J(0) = 0$ by one-homogeneity and, if $p \in \partial J(0)$, we have
	\begin{equation*}
	J(v) \ge \inner{p, v}, \text{ for all $v \in X$.}
	\end{equation*}
	Summing and subtracting $J(u) = \inner{p,u}$ we obtain the subdifferential inequality. 
	
	For the converse inclusion, take some $p \in \partial J(u)$, and taking $v = 0$ in the sub-differential inequality, we have 
	\begin{equation*}
	J(u) \le \inner{p, u}.
	\end{equation*} 
	Now it suffices to prove that $\inner{p, v} \le J(v)$ for all $v \in X$, as this is equivalent to $p \in \partial J(0)$ and implies the equality $J(u) = \inner{p,u}$. By contradiction, if there is some $v$ such that $J(v) - \inner{p, v} < 0$, take $\lambda v$ instead of $v$ for $\lambda>0$, then by the one-homogeneity we have 
	\begin{equation*}
	\lambda\underbrace{\left(J(v) - \inner{p, v}\right)}_{< 0} \ge J(u) - \inner{p, u}.
	\end{equation*}
	Taking $\lambda$ large enough we arrive at a contradiction.  
	
	For the second statement, notice that it suffices to prove that $C \subset \partial J(0) \subset \overline{C}$, since $\partial J(0)$ being a subdifferential implies that it is already weak-$\star$ closed and convex. 
	
	The first inclusion follows from the definition; we pass to the second. Suppose there is some $p_0 \in \overline{C}\setminus \partial J(0)$, then $\{p_0\}$ is compact, convex and disjoint from $\partial J(0)$. Using the convex separation Theorem in $X^\star$, see~\cite[Chap.~2.5 Prop.~4]{bourbaki2003topological} or ~\cite[Problem 9, page 447]{brezis2011functional}, we can find a continuous linear functional in the weak-$\star$ topology, that is some $v \in X$, that strictly separates $\{p_0\}$ and $\partial J(0)$ giving 
	\begin{equation*}
		\inner{p_0, v} > \sup_{p \in \partial J(0)} \inner{p_0, v} \ge J(v).
	\end{equation*}
	We then arrive at a contradiction. 
\end{proof}
}\fi

Recalling the definition of the $\TV$ functional as
\begin{equation}\label{TV_setK}
	\TV(u) = \sup_{p \in \mathcal{K}} \int_{\Omega}p(x)u(x)\dd x, \text{ where }
	\mathcal{K}:=
	\left\{
		p = \text{div}\phi: 
		\begin{array}{c}
			\phi \in C^1_c(\Omega; \mathbb{R}^N)\\
			\norm{\phi}_{\infty} \le 1
		\end{array}
	\right\},
\end{equation}
from Lemma~\ref{lemma.subdifferential_one_homo_functional}  
$\partial_{L^2} \TV(0) = \overline{\mathcal{K}}$,
with the closure being taken with respect to the weak topology of $L^2(\Omega)$. It then follows that, for $\Omega$ convex and bounded, the subdifferential of $\TV$ in $L^2$ assumes the form
\begin{equation*}
	\partial \TV(u) = 
	\left\{
		p = -\ddiv (z): 
		\begin{array}{c}
			z \in H^1_0(\ddiv, \Omega), \ \norm{z}_{\infty} \le 1 \\
			\TV(u) = \int_{\Omega}p(x)u(x)\dd x
		\end{array}
	\right\},
\end{equation*}
where $H^1_0(\ddiv; \Omega)$ denotes the closure of $C^\infty_c(\Omega; \mathbb{R}^d)$ with respect to the norm $\norm{z}_{H^1(\ddiv)}^2 = \norm{z}_{L^2(\Omega)}^2 + \norm{\ddiv z}_{L^2(\Omega)}^2$.
\if{
\begin{lemma}[\cite{bredies2016pointwise}]\label{lemma.closure_setK}
	For $1\le p < \infty$, the closure of $\mathcal{K}$ in the weak topology of $L^q(\Omega)$ is given by
	\[
		\overline{\mathcal{K}} = K \eqdef
		\left\{
			\ddiv z: 
			\begin{array}{c}
				z \in W^q_0(\ddiv ;\Omega)\\
				\norm{z}_{\infty} \le 1
			\end{array}
		\right\}, 
		\text{ where $\frac{1}{p} + \frac{1}{q} = 1$.}
	\]
\end{lemma}
\begin{proof}
	It is enough to prove that $\mathcal{K} \subset K \subset \overline{\mathcal{K}}$ and that $K$ is a closed set. The first inclusion follows from the definition; let us first prove that $K$ is closed. Since $K$ is convex, it is weakly closed if and only if it is strongly closed. 
	
	First consider $1<p$ and $q < \infty$. Take $\left(\text{div} z_n\right)_{n \in \mathbb{N}} \subset K$ converging in $L^q(\Omega)$ to $h$. Since $\norm{z_n}_{\infty} \le 1$ and $\Omega$ is bounded, this sequence is also bounded in $L^q(\Omega)$, and we can extract a weakly convergent subsequence, $z_{n_k} \rightharpoonup z$. Then, for all $\phi \in C^{\infty}_c(\Omega; \mathbb{R}^N)$, we have
	\begin{align*}
		\int_{\Omega}z\cdot \nabla \phi\dd x 
		= \lim_{k \to \mathbb{N}} 
		\int_{\Omega} z_{n_k}\cdot \nabla \phi\dd x 
		= \lim_{k \to \mathbb{N}} 
		-\int_{\Omega} \text{div}z_{n_k} \phi\dd x 
		= -\int_{\Omega} h\phi\dd x.
	\end{align*} 
	So $h = \text{div}z \in L^q(\Omega)$ and $z \in W^q(\text{div};\Omega)$. 
	
	To prove that $z \in W^q_0(\text{div};\Omega)$, consider the set 
	\[
	\left\{
		(f, \text{div}f): f \in W^q_0(\text{div};\Omega), \ \norm{f}_{\infty} \le 1
	\right\} \subset L^q(\Omega; \mathbb{R}^{N+1}),
	\]
	which is convex and strongly closed, hence also weakly closed. Hence, as $z_{n_k}\in W^p(\ddiv; \Omega)$ and $\left(z_{n_k}, \text{div}z_{n_k}\right)_{k \in \mathbb{N}}$ converges weakly, the conclusion follows. For the case $p=1,q = \infty$ we can instead take a subsequence of $\left(z_n\right)_{n \in \mathbb{N}}$ converging in the weak-$\star$ topology and use analogous arguments.  

	Finally, we must show that $K \subset \overline{\mathcal{K}}$. For this, it suffices to show that, for every $v \in L^p(\Omega)$
	\[
		\int_{\Omega}p(x) v(x) \dd x \le \TV(v), \text{ for all $p \in K$}.
	\]
	Given some $v \in L^p(\Omega)\cap$BV$(\Omega)$, thanks to Meyers-Serrin's Theorem~\cite{ambrosio2000functions} we can take a sequence $v_n$ converging to $v$ in $L^p(\Omega)$ and such that 
	$
		\TV(v_n) \xrightarrow[n \to \infty]{} \TV(v). 
	$
	Hence
	\begin{align*}
		\int_{\Omega}v\text{div}z\dd x 
		= \lim_{n \to \infty} 
		\int_{\Omega}v_n\text{div}z\dd x
		= \lim_{n \to \infty} 
		-\int_{\Omega}\nabla v_n\cdot z\dd x
		\le \lim_{n \to \infty} 
		\int_{\Omega}|\nabla v_n|\dd x = \TV(v).
	\end{align*}
	The result follows. 
\end{proof}
\begin{corollary}\label{corollary.subdiff_TV}
	For $1\le p < \infty$, the subdifferential $\partial_{L^p} \TV(u) \subset L^q(\Omega)$ is given by
	\begin{equation*}
	\partial \TV(u) = 
	\left\{
		p = \ddiv (z): 
		\begin{array}{c}
			z \in W^p_0(\ddiv, \Omega), \ \norm{z}_{\infty} \le 1 \\
			\TV(u) = \int_{\Omega}p(x)u(x)\dd x
		\end{array}
	\right\}. 
	\end{equation*}
\end{corollary}
}\fi

Next we prove another property of the subdifferential of $\TV$ based on the coarea formula (see~\cite[Thm.~3.40]{ambrosio2000functions}): for any $u \in \BV(\Omega)$ it holds that 
\begin{equation}\label{eq.coarea_formula}
    \TV(u) = \int_{\mathbb{R}}\Per(\{u \ge t\})\dd t. 
\end{equation}
\begin{lemma}\label{lemma.subdifferential_h_lemma}
	For any $u \in \BV(\Omega)$ and $p \in \partial \TV(u)$, it holds that 
	\begin{equation*}
		p \in \partial \TV(u^+),\quad
		-p \in \partial \TV(u^-).
	\end{equation*}
\end{lemma}
\begin{proof}
	We can show that for a.e. $s \in \mathbb{R}$ it holds that $\Per\left(\{u > s\}\right) = \Per\left(\{u \ge s\}\right) = \Per\left(\{u \ge s\}^c\right)$. In addition, by Lemma~\ref{lemma.subdifferential_one_homo_functional}, if $p \in \partial \TV(u)$, one has
	\begin{align*}
	\TV(u) 
    &= \int_{0}^{+\infty} \Per(\{u^+ > s\})\dd s + \int_{0}^{+\infty} \Per(\{u^- > s\})\dd s\\
	= 
    \int_{\Omega}p(x)u(x)\dd x 
    &= \int_{0}^{+\infty}\int_{\{u^+ > s\}}p\dd x \dd s- 
	\int_{0}^{+\infty}\int_{\{u^- > s\}}p\dd x \dd s.
	\end{align*} 
	In particular, we have 
	\begin{equation*}
	\int_0^{+\infty} \underbrace{
		\left(\Per(\{u^+ > s\}) - \int_{\{u^+ > s\}} p(x)\dd x\right)
	}_{\ge 0}\dd s = \int_0^{+\infty} \underbrace{
		\left(-\Per(\{u^- > s\}) - \int_{\{u^- > s\}} p(x)\dd x\right)
	}_{\le 0}\dd s.
	\end{equation*}
	Since the integrands on each side have constant and opposite signs, for a.e.~$s \in \mathbb{R}$ 
	\[
	    \Per(\{u^+ > s\}) = \int_{\{u^+ > s\}} p(x)\dd x \text{ and } 
	    \Per(\{u^- > s\}) = \int_{\{u^- > s\}} -p(x)\dd x.
	\]
	Integrating over $[0,+\infty)$, the coarea formula gives
	\begin{equation*}
		\TV(u^+) = \int_{\Omega} p(x)u^{+}(x)\dd x \text{ and } \TV(u^-) = \int_{\Omega} -p(x)u^{-}(x)\dd x,
	\end{equation*}
	and the result follows from Lemma~\ref{lemma.subdifferential_one_homo_functional}. 
\end{proof}

\if{
After the discussion from section \ref{corollary.subdiff_TV}, we conclude that if $u$ solves \eqref{ROF}, then there exists some $z \in H^1_0(\text{div}, \Omega)\cap L^{\infty}(\Omega)$ such that 
\begin{equation}
\left\{
\begin{array}{ll}
	-\lambda\text{div}z(x) + u(x) = g(x),& \text{ a.e. in $\Omega$}\\
	|z(x)| \le 1,& \text{ a.e. in $\Omega$}\\
	z\cdot\nu = 0,& \text{ on $\partial\Omega$}\\
	\TV(u) = \int_{\Omega}-\text{div}z(x)u(x)\dd x.
\end{array}
\right.
\end{equation}

A few comments are necessary regarding these optimality conditions.
\begin{enumerate}
	\item For the purpose of solving (ROF), the natural choice is to define $\TV$ as a functional over $L^2(\Omega)$, hence the vector field $z$ is taken in $H^1_0(\text{div}, \Omega)$.
	\item The boundary conditions $z\cdot \nu = 0$ are understood in the sense of the boundary trace. 
	\item The final condition however can be strengthened. Indeed, it is a global characterization of the vector field $z$, however pointwise characterizations of the subdifferential of $\TV$ have been studied by Anzellotti in \cite{anzellotti1983pairings} and further developed in \cite{bredies2016pointwise} and others.
	
	Given a function $u \in$BV$(\Omega)$, he defined a trace operator $z \mapsto [z, Du]$ such that
	\[
		\int_{\Omega} [z,Du]\phi \dd|Du| = - \int_{\Omega} u \text{div}(\phi z)\dd x = \int_{\Omega} \phi z \cdot \dd Du. 
	\]
	It is then proven that 
	\[
	p \in \partial \TV(u) \Leftrightarrow 
	\begin{array}{c}
	\exists \ z \in W^p_0(\text{div}, \Omega), \ \norm{z}_{\infty} \le 1\\
	p = -\text{div}z, \ [z,Du] = 1 \text{ in $L^1(\Omega; |Du|)$}.
	\end{array}
	\]
	In combination with the definition of this trace operator, it is equivalent to $z\cdot Du = |Du|$. 
	\item The previous point implies that, whenever $u > 0$ is smooth enough, we can formally interpret the Euler-Lagrange equation as 
	\[
		-\lambda\text{div}\frac{\nabla u(x)}{|\nabla u(x)|} + u(x) = g(x). 
	\]
\end{enumerate}
}\fi

We finish this appendix with a property of the solutions of~\eqref{ROF}: if $u$ solves~\eqref{ROF}, then its positive part solves the same problem with an additional positivity constraint (see also~\cite[Lemma~A.1]{chambolle2004algorithm}). 
\begin{theorem}\label{theorem.ROF_positive_solution}
	A function $u$ is the solution of~\eqref{ROF} if, and only if,
    \[
      u-g \in \lambda \partial_{L^2}\TV(u).  
    \]
    In addition, it holds that
	\begin{equation}
	u^+ \in \argmin_{v \ge 0} \TV(v) + \frac{1}{2\lambda}\int_{\Omega} |v(x) - g(x)|^2\dd x. 
	\end{equation}
\end{theorem}
\begin{proof}
    The first property is a direct consequence of Fermat's rule and the fact that the $L^2$ norm is smooth. For the second, since the constrained problem remains strictly convex, it suffices to show that 
	\begin{equation}
		g - u^+ \in \partial_{L^2}\left(\lambda \TV + \chi_{v \ge 0}\right)(u^+).
	\end{equation}
	Writing $u = u^+ - u^-$ with both $u^+, u^- \ge 0$, from Lemma~\eqref{lemma.subdifferential_h_lemma} and the Euler-Lagrange equation of~\eqref{ROF}:
    \[
        g-u \in \lambda\partial_{L^2}\TV(u),    
    \]
    we know that 
	\begin{equation*}
		g - u^+ = -u^- + \lambda p, \text{ with $p \in \partial_{L^2} \TV(u^+)$}.
	\end{equation*}
	So it suffices to show that $-u^- + \lambda \partial_{L^2}\TV(u^+) \subset \partial_{L^2}\left(\lambda \TV + \chi_{v \ge 0}\right)(u^+)$. Take some $v \ge 0$ and since $p \in \partial_{L^2} \TV(u^+)$, we have 
	\[
		\lambda\left(\TV(v) - \TV(u^+) - \inner{p, v - u^+}\right) \ge 0 \ge - \inner{u^-, v}. 
	\]
	Rearranging the terms we obtain the desired relation $\lambda p - u^- \in \partial_{L^2}\left(\lambda \TV + \chi_{v \ge 0}\right)(u^+)$. 
\end{proof}

\section{N-dimensional example}\label{appendix.Ndim}
In this appendix we prove the theoretical characterization of the optimal radius used in the first experiment of Section~\ref{section.experiments}.

\begin{lemma}[\cite{carlier2019total}, Lemma 2.1]\label{lemma.sufficient_condition}
	Let $\rho_0 \in \mathcal{P}_2(\mathbb{R}^d)$ and take $\Omega = \mathbb{R}^d$. If $\rho_1 \in \text{BV}\left(\mathbb{R}^d\right)\cap \mathcal{P}_2(\mathbb{R}^d)$ and there exists $p \in \partial \TV(\rho_1)$ satisfying
	\begin{equation}
		\frac{\psi}{\tau} + p \ge 0, \text{ with equality $\rho_1$-a.e.}
	\end{equation}
	then $\rho_1$ solves~\eqref{TV-W}. 
\end{lemma}
\begin{proof}
	By definition, for any $\rho \in \text{BV}(\Omega)\cap \mathcal{P}_2(\mathbb{R}^d)$, one has 
	\[
		\TV(\rho) \ge \int_{\mathbb{R}^d}p\dd \rho,
	\]
	So using Kantorovitch duality and taking $(\varphi, \psi)$ the potentials between $\rho_0$ and $\rho_1$, for every $\rho$ one obtains
	\begin{align*}
		\frac{1}{2}W^2_2(\rho_0, \rho) 
		&\ge \int\varphi \dd\rho_0 + \int \psi \dd \rho 
		=   \int\varphi \dd\rho_0 + \int \psi \dd \rho_1 + \int \psi \dd(\rho - \rho_1)\\
		&= \frac{1}{2}W_2^2(\rho_0, \rho_1) + \int \psi \dd(\rho - \rho_1)\\
		&\ge \frac{1}{2}W_2^2(\rho_0, \rho_1) + \tau\left(\TV(\rho_1) - \int p\dd\rho\right)
	\end{align*}
	Which, by the definition of $\TV$, implies that
	\begin{equation*}
		\TV(\rho_1) + \frac{1}{2\tau}W^2_2(\rho_0, \rho_1) \le \TV(\rho) + \frac{1}{2\tau}W^2_2(\rho_0, \rho),
	\end{equation*}
	for all $\rho \in \text{BV}(\Omega)\cap \mathcal{P}_2(\Omega)$, and therefore $\rho_1$ is a minimizer of \eqref{TV-W}.
\end{proof}

Given an initial radius $r_0$, we set
\begin{equation}
	\mathcal{E}_{\tau, r_0}(\rho) := \frac{1}{2\tau}W^2_2(\rho_{r_0}, \rho) + \TV(\rho).
\end{equation} 
We minimize $\mathcal{E}_{\tau, r_0}(\rho_{r_1})$ as a real valued function of $r_1$ and check that the minimizer $\rho_{r_1}$ satisfies the sufficient condition~\ref{lemma.sufficient_condition} with a well tailored $z$. 

The first step is to find the optimal transport map (if it exists). From~\cite[Thm.~1.48]{santambrogio2015optimal} it suffices to find a map $T_\sharp \rho_{r_0} = \rho_{r_1}$ that can be written as the gradient of a convex function. It is easy to check that $T = \frac{r_1}{r_0}\text{id}$ is the gradient of $u(x) := \frac{r_1}{2r_0}|x|^2$ and $\rho_{r_1} = T_{\sharp}\rho_{r_0}$.
\if{
\begin{align*}
	\int_{\mathbb{R}^d}\phi \dd(T_{\sharp}\rho_{r_0}) 
	&= \int_{B(0,r_0)}\phi\left(\frac{r_1}{r_0}x\right) \frac{\dd x}{\omega_dr_0^d}\\
	&= \int_{B(0,r_1)}\phi\left(y\right) \frac{1}{\omega_dr_0^d}\left(\frac{r_0}{r_1}\right)^d\dd y\\
	&= \int_{\mathbb{R}^d}\phi(y)\dd \rho_{r_1}.
\end{align*}
}\fi

In the sequel, let us compute $\mathcal{E}_{\tau, r_0}(\rho_{r_1})$ for some $r_1$. The Wasserstein term can be easily computed using the optimal map, namely 
\begin{align*}
	\frac{1}{2\tau}W^2_2\left(\rho_{r_0}, \rho_{r_1}\right)
	&= \frac{1}{2\tau}\int_{B(0,r_0)} \left|x - \frac{r_1}{r_0}x\right|^2\frac{\dd x}{\omega_dr^d_0}
	= \frac{1}{2\tau}\left(1 - \frac{r_1}{r_0}\right)^2\frac{1}{\omega_dr^d_0}\int_{B(0,r_0)} \left|x\right|^2\dd x\\
	&= \frac{1}{2\tau}\left(1 - \frac{r_1}{r_0}\right)^2\frac{1}{\omega_dr^d_0}\int_0^{r_{0}} r^2\mathcal{H}^{d-1}(\partial B(0,r))\dd r\\
	&= \frac{1}{2\tau}\left(1 - \frac{r_1}{r_0}\right)^2\frac{1}{\omega_dr^d_0}\int_0^{r_{0}} 2\pi \omega_{d-1} r^{d+1}\dd r\\
	&= \frac{1}{2(d+2)\tau}\frac{2\pi \omega_{d-1}}{\omega_d}(r_1 - r_0)^2.
\end{align*} 

To compute the total variation term, we will use the coarea formula. 
\begin{align*}
	\TV(\rho_{r_1}) 
	&= \int_{\mathbb{R}} \text{Per}\left(\{\rho_{r_1} > s\}\right)\dd s
	= \int_0^{1/\omega_dr_1^d} \text{Per}\left(\{\rho_{r_1} > s\}\right)\dd s\\
	&= \frac{1}{\omega_dr_1^d} \text{Per}\left(B(0, r_1)\right)
	= \frac{1}{\omega_dr_1^d} \mathcal{H}^{d-1}\left(\partial B(0, r_1)\right)\\
	&= \frac{2\pi \omega_{d-1}}{\omega_d}\frac{1}{r_1}.
\end{align*}
Hence, setting $K_d:= \frac{2\pi \omega_{d-1}}{\omega_d}$ we obtain 
\begin{equation*}
	\mathcal{E}_{\tau, r_0}(r_1) = \frac{K_d}{2(d+2)\tau}(r_1 - r_0)^2 + \frac{K_d}{r_1}, 
\end{equation*}
which is minimized by the positive root of 
\begin{equation}
	r_1^2(r_1 - r_0) = r_0^2(d+2)\tau .
\end{equation}

Now, supposing that some function $z$ satisfying the conditions of Lemma~\ref{lemma.sufficient_condition} 
exists, let us try to find it explicitly. Starting with the Kantorovitch potential $\psi$, we know that $T = \frac{r_0}{r_1}\text{id} = \text{id} - \nabla \psi$, which means that $\psi$ is of the form
\begin{equation*}
	\psi(x) = \frac{r_1 - r_0}{2r_1}|x|^2 + C. 
\end{equation*}
Hence, we look for $z$ such that
\[
	\ddiv z(x) = \frac{r_0 - r_1}{2\tau r_1}|x|^2 - \frac{C}{\tau}, \text{ for all $x \in B(0,r_1)$},
\]
and the constant $C$ can be computed explicitly with the relation 
\[
	\TV(\rho_{r_1}) = \frac{K_d}{r_1} = \int_{B(0,r_1)}\ddiv z\dd \rho_{r_1}.
\]
\if{namely 
\begin{align*}
	\frac{K_d}{r_1}
	&= \frac{1}{\omega_d r_1^d}\int_{B(0,r_1)}\left(-\frac{\psi}{\tau}\right)\dd x\\
	&= -\frac{C}{\tau} + \frac{1}{\tau \omega_d}\frac{r_0 - r_1}{2r_1^{d+1}}\int_{B(0, r_1)}|x|^2\dd x \\
	&= -\frac{C}{\tau} + \frac{1}{\tau \omega_d}\frac{r_0 - r_1}{2r_1^{d+1}}\int_0^{r_1}r^2\mathcal{H}^{d-1}(\partial B(0,r_1))\dd r\\
	&= -\frac{C}{\tau} + \frac{K_d}{2r_1},
\end{align*}
which gives $C = -\frac{K_d\tau}{2r_1}$ and over $B(0, r_1)$, $\ddiv(z)$ assumes the form 
\begin{equation*}
	\ddiv(z) = -\frac{r_1 - r_0}{2\tau r_1}|x|^2 + \frac{K_d}{2r_1}.
\end{equation*}
}\fi

In particular, we can take $z$ of the form 
\begin{equation*}
	z(x) = -\frac{r_1 - r_0}{2\tau r_1}x^3 + \frac{2}{r_1}x + z(0), \text{ for $x \in B(0,r_1)$},
\end{equation*}
and $x^3$ stands for the vector $\left(x_i^3\right)_{i = 1}^d$. The condition $\norm{z}_{L^\infty(B(0,r_1))} \le 1$ also holds for a suitable choice of $z(0)$ and to conclude it suffices to extend $z$ outside $B(0,r_1)$ in order to keep this bound and have compact support. As such extension always exists, we conclude.

%
%
\bibliographystyle{plain}
\bibliography{references.bib}

\end{document}